\documentclass[12pt, reqno]{amsart}

\usepackage{amsmath, amssymb, amscd, amsthm, comment, amsxtra}
\usepackage{graphicx}
\usepackage{subfigure}
\usepackage{enumerate}
\usepackage[pdftex, linktocpage=true]{hyperref}
\usepackage{euscript}
\usepackage[format=plain,labelfont=bf,up,width=.99\textwidth]{caption}
\usepackage[toc,page]{appendix}
\usepackage{mathabx}

\theoremstyle{plain}
\newtheorem{thm}{Theorem}[section]
\newtheorem*{mainthm}{Main Theorem}
\newtheorem{lem}[thm]{Lemma}
\newtheorem{prop}[thm]{Proposition}

\theoremstyle{definition}
\newtheorem{defn}[thm]{Definition}

\newtheorem{remark}[thm]{Remark}

\newcommand{\diam}{\operatorname{diam}}

\newcommand{\dist}{\operatorname{dist}}

\newcommand{\Jac}{\operatorname{Jac}}

\numberwithin{equation}{section}
\newcommand{\thmref}[1]{Theorem~\ref{#1}}

\newcommand{\lemref}[1]{Lemma~\ref{#1}}

\newcommand{\figref}[1]{Figure~\ref{#1}}

\newcommand{\cR}{{\mathcal R}}

\newcommand{\cD}{{\mathcal D}}

\newcommand{\bfR}{{\mathbf R}}
\newcommand{\bbN}{{\mathbb N}}
\newcommand{\bbZ}{{\mathbb Z}}
\newcommand{\bbQ}{{\mathbb Q}}
\newcommand{\bbR}{{\mathbb R}}
\newcommand{\bbC}{{\mathbb C}}
\newcommand{\bbD}{{\mathbb D}}

\newcommand{\bfx}{{\mathbf x}}
\newcommand{\bfq}{{\mathbf q}}
\newcommand{\bfp}{{\mathbf p}}

\title[Renormalization in the Golden-Mean Semi-Siegel H\'enon Family]{Renormalization in the Golden-Mean Semi-Siegel H\'enon Family: Non-Quasisymmetry}
\author{Jonguk Yang}

\begin{document}

\begin{abstract}
For quadratic polynomials of one complex variable, the boundary of the golden-mean Siegel disk must be a quasicircle. We show that the analogous statement is not true for quadratic H\'enon maps of two complex variables.
\end{abstract}

\maketitle

\section{Introduction}

Let $\theta \in (\bbR \setminus \bbQ )/\bbZ$ be an irrational rotation number. Then $\theta$ can be represented by an infinite continued fraction:
$$
\theta = [a_0, a_1, \ldots{}] = \cfrac{1}{a_0+\cfrac{1}{a_1+\cfrac{1}{a_2 + \ldots{}}}}.
$$
The {\it $n$th partial convergent of $\theta$} is the rational number
$$
\frac{p_n}{q_n} = [a_0, a_1, \ldots , a_n].
$$
The denominator $q_n$ is called the $n$th {\it closest return moment}. The sequence $\{q_n\}_{n=0}^\infty$ satisfy the following inductive relation:
$$
q_0 = 1
\hspace{5mm} , \hspace{5mm}
q_1 =a_0
\hspace{5mm} \text{and} \hspace{5mm}
q_{n+1} = a_n q_n + q_{n-1}
\hspace{5mm} \text{for} \hspace{5mm}
n \geq 1.
$$
We say that $\theta$ is {\it Diophantine of order $d$} if there exists $C >0$ such that
$$
q_{n+1}< C q_n^{d-1}
\hspace{5mm} \text{for all} \hspace{5mm}
n \in \mathbb{N}
$$
If $\theta$ is Diophantine of order $2$, then $\theta$ is said to be of {\it bounded type}. It is easy to show that $\theta$ is of bounded type if and only if $a_n$'s are uniformly bounded (see e.g. \cite{M}). The simplest example of a bounded type rotation number is the inverse-golden mean:
$$
\theta_* = \frac{\sqrt{5}-1}{2} = [1, 1, \ldots].
$$
Observe that for $\theta_*$, the sequence of closest return moments $\{q_n\}_{n=0}^\infty$ is the {\it Fibonacci sequence}.

Consider the standard one-parameter family of quadratic polynomials
$$
f_c(z) := z^2 +c \hspace{5mm} \text{for } c \in \bbC.
$$
This is referred to as the {\it quadratic family}. A quadratic polynomial $f_c$ is determined uniquely by the multiplier $\mu \in \bbC$ at a fixed point $x_c \in \bbC$ for $f_c$. In fact, we have
\begin{equation}\label{eq:1dmulti}
c=\frac{\mu}{2}-\frac{\mu^2}{4}.
\end{equation}
Let $c_0$ be the unique parameter such that for $f_{c_0}$, the fixed point $x_{c_0}$ is irrationally indifferent with rotation number $\theta$ (i.e. $\mu = e^{2\pi i \theta}$). If $\theta$ is equal to the inverse golden-mean $\theta_*$, then we denote $f_{c_0}$ as simply $f_*$. 

The quadratic polynomial $f_{c_0}$ is said to be {\it Siegel} if it is locally linearizable at $x_0$. More precisely, $f_{c_0}$ is Siegel if there exist neighborhoods $U$ of $0$ and $V$ of $x_0$, and a biholomorphic change of coordinates
$$
\psi : (U, 0) \to (V, x_0)
$$
such that 
$$
\psi^{-1} \circ f_{c_0} \circ \psi(z) = \mu z.
$$
A classic theorem of Siegel states, in particular, that $f_{c_0}$ is Siegel whenever $\theta$ is Diophantine. Moreover, it is known that the linearizing map $\psi$ can be biholomorphically extended to a map $\psi: (\mathbb{D}, 0) \rightarrow (\Delta, x_0)$ so that its image $\Delta := \psi(\mathbb{D})$ is maximal (see e.g. \cite{M}). We call $\Delta$ and $\partial \Delta$ the {\it Siegel disk} and the {\it Siegel boundary} of $f_{c_0}$ respectively. See \figref{fig:siegeldisk}.

It is natural to ask whether $\partial \Delta$ is a Jordan curve, and if so, whether it is quasisymmetric, or even smooth. The following theorem settles these questions in the case when $\theta$ is of bounded type (see \cite{He}):

\begin{thm}[Douady, Ghys, Herman, Shishikura]\label{siegel boundary rotation}
Suppose that $\theta$ is of bounded type. Then $f_{c_0}$ has its critical point $0$ on its Siegel boundary $\partial \Delta$, and the restriction $f_{c_0}|_{\partial \Delta} : \partial \Delta \to \partial \Delta$ is quasisymmetrically conjugate to the rigid rotation of the unit circle $\partial \mathbb{D}$ by $\theta$.
\end{thm}

\begin{figure}[h]
\centering
\includegraphics[scale=0.25]{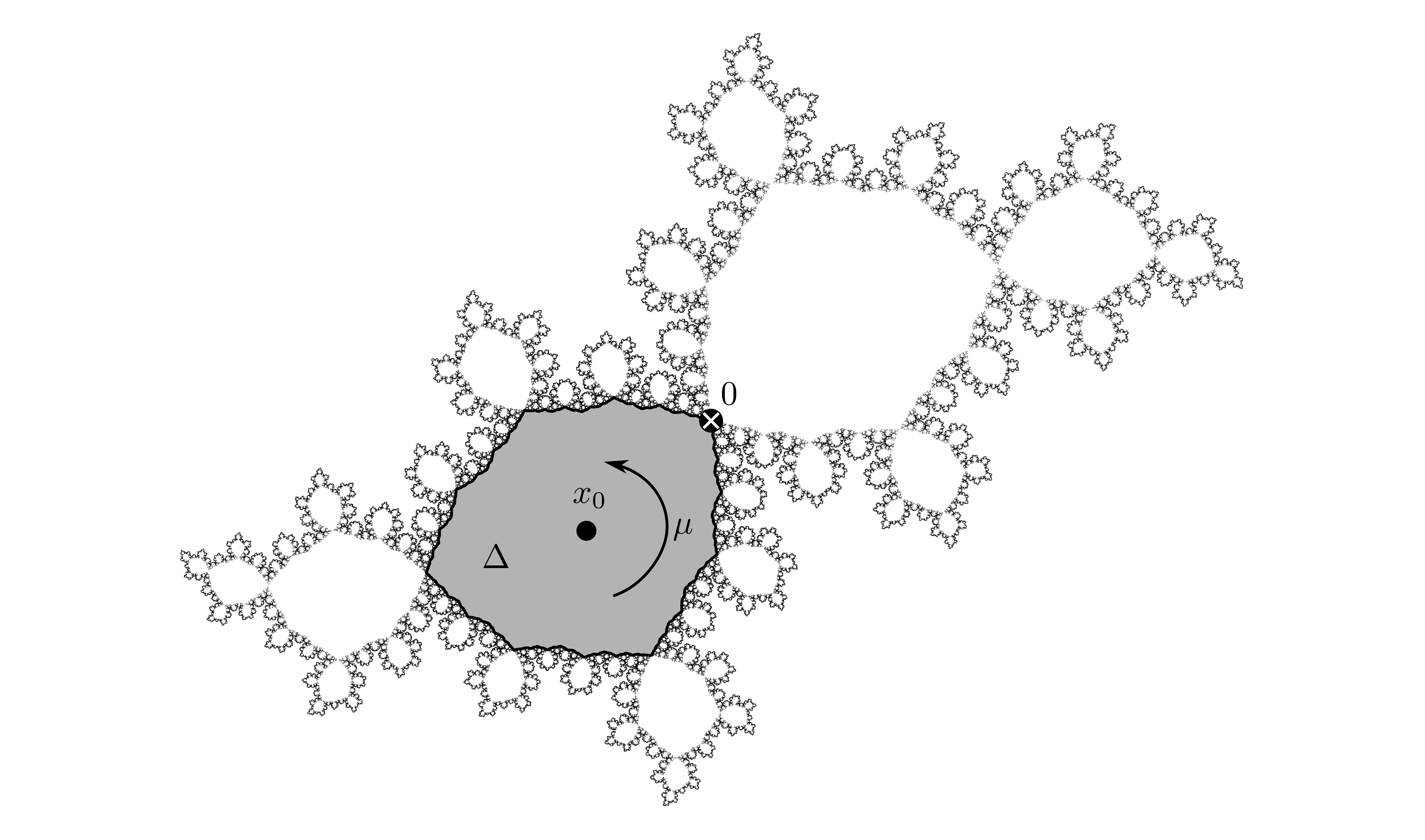}
\caption{The Siegel disk $\Delta$ of the golden-mean Siegel quadratic polynomial $f_*$. The critical point $0$ is on $\partial \Delta$, and the restriction $f_*|_{\partial \Delta} : \partial \Delta \to \partial \Delta$ is quasisymmetrically conjugate to the rigid rotation of the unit circle $\partial \mathbb{D}$ by the angle $\theta_* = (\sqrt{5}-1)/2$.}
\label{fig:siegeldisk}
\end{figure}

From Theorem \ref{siegel boundary rotation}, it immediately follows that $\partial \Delta$ cannot be smooth if $\theta$ is of bounded type, since any curve containing the critical point cannot be both invariant and smooth. It is important to note, however, that there are examples of quadratic Siegel polynomials for which the Siegel boundary does not contain the critical point and is in fact smooth (see \cite{BuCh}).

The main goal of this paper is to carry the study of Siegel boundaries to a higher dimensional setting. To this end, consider the following two-dimensional extension of the quadratic family:
$$
H_{c,b}(x,y):=(x^2+c - by,x) \hspace{5mm} \text{for } c \in \bbC \text{ and } b \in \bbC \setminus \{0\}.
$$
This is referred to as the {\it (complex quadratic) H\'enon family}.

It is easy to see that $H_{c,b}$ has constant Jacobian:
$$
\text{Jac} \, H_{c,b} \equiv b.
$$
Moreover, for $b=0$, the map $H_{c,b}$ degenerates to the following embedding of $f_c$:
$$
\iota(f_c)(x,y) := (f_c(x),x).
$$
Hence, the parameter $b$ measures how far $H_{c,b}$ is from being a degenerate one-dimensional system. In this paper, we will always assume that $H_{c,b}$ is a dissipative map (i.e. $|b|<1$). 

A H{\'e}non map $H_{c,b}$ is determined uniquely by the multipliers $\mu \in \bbC \setminus \{0\}$ and $\nu \in \bbD \setminus \{0\}$  at a fixed point $\mathbf{x}_{c,b} \in \bbC^2$. In fact, we have
$$
b = \mu\nu,
$$
and
$$
c=(1+\mu\nu)\left(\frac{\mu}{2}+\frac{\nu}{2}\right)-\left(\frac{\mu}{2}+\frac{\nu}{2}\right)^{2}.
$$
Compare with \eqref{eq:1dmulti}. For any Jacobian $b \in \mathbb{D} \setminus \{0\}$, there exists a unique parameter $c_b \in \bbC$ such that one of the multipliers $\mu$ of the fixed point $\mathbf{x}_b := \mathbf{x}_{{c_b}, b}$ for $H_b := H_{c_b, b}$ is given by
$$
\mu=e^{2\pi i \theta}.
$$
Note that in this case, we have $|\nu| = |b|$.

The H\'enon map $H_b$ is said to be {\it semi-Siegel} if it is locally linearizable at $\mathbf{x}_b$. More precisely, $H_b$ is semi-siegel if there exist neighborhoods $\mathbf{U}$ of $(0,0)$ and $\mathbf{V}$ of $\mathbf{x}_0$, and a biholomorphic change of coordinates $\Psi_b : (\mathbf{U}, (0,0)) \to (\mathbf{V}, \mathbf{x}_b)$ such that 
$$
\Psi_b^{-1} \circ H_b \circ \Psi_b(x,y) = (\mu x, \nu y).
$$
Similar to the one-dimensional case, $H_b$ is semi-Siegel whenever $\theta$ is Diophantine. Furthermore, the linearizing map $\Psi_b$ can be biholomorphically extended to a map $\Psi_b : (\mathbb{D} \times \bbC, (0,0)) \rightarrow (\mathcal{C}_b, \mathbf{x}_b)$ so that its image $\mathcal{C}_b := \Psi_b(\mathbb{D} \times \bbC)$ is maximal (see \cite{MoNiTaUe}). We call $\mathcal{C}_b$ the {\it Siegel cylinder} of $H_b$. In the interior of $\mathcal{C}_b$, the dynamics of $H_b$ is conjugate to rotation by $\theta$ in one direction, and compression by $\nu$ in the other direction. Clearly, the orbit of every point in $\mathcal{C}_b$ converges to the analytic disk $\mathcal{D}_b:=\Psi_b(\mathbb{D} \times\{0\})$ at height $0$. We call $\mathcal{D}_b$ the {\it Siegel disk} of $H_b$. See \figref{fig:siegelcylinder}

\begin{figure}[h]
\centering
\includegraphics[scale=0.25]{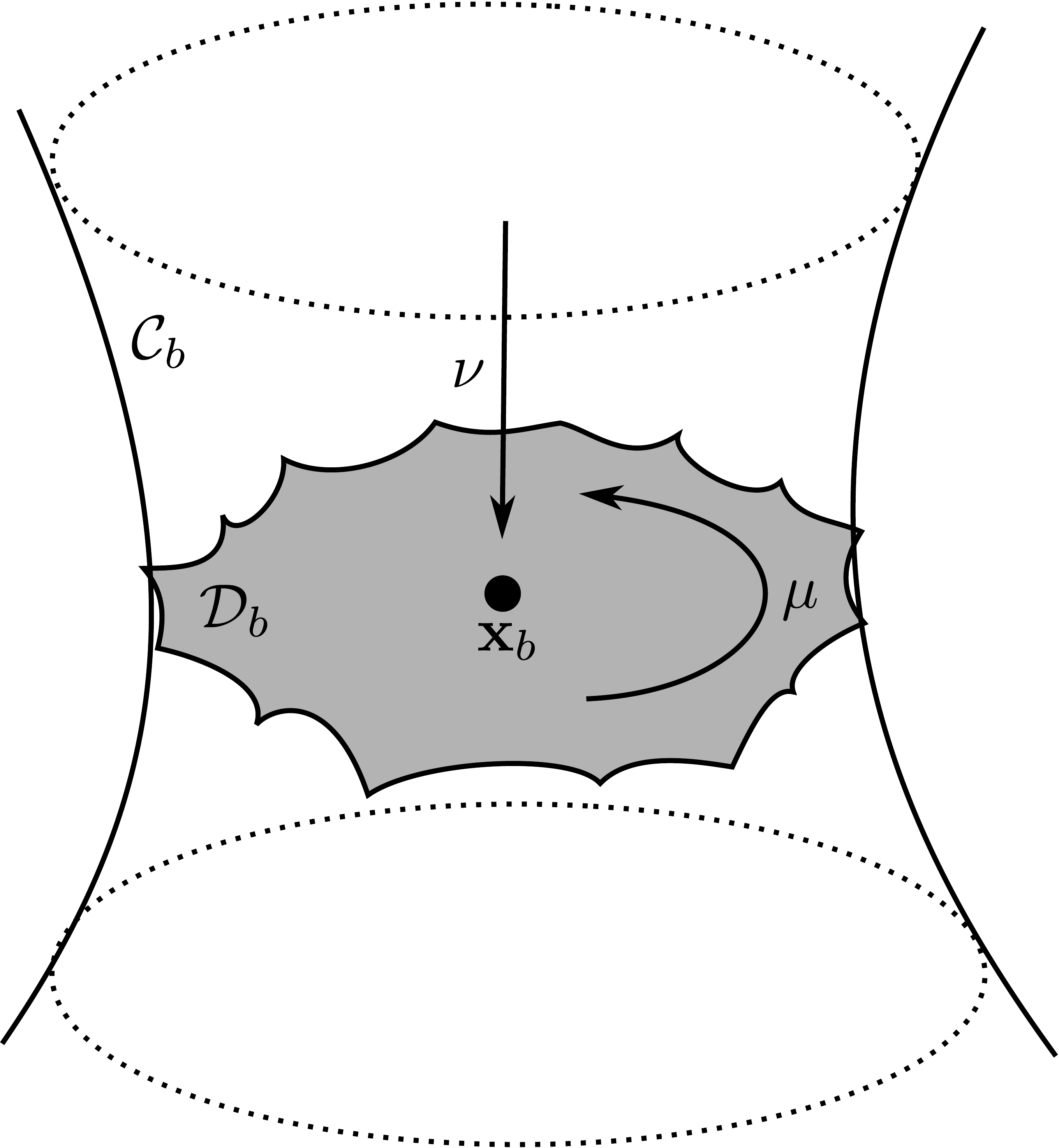}
\caption{The Siegel cylinder $\mathcal{C}_b$ and the Siegel disk $\mathcal{D}_b$ of $H_b$.}
\label{fig:siegelcylinder}
\end{figure}

The study of semi-Siegel H\'enon maps had been a wide open subject until a recent work of Gaidashev, Radu, and Yampolsky (see \cite{GaRaYam}), who proved:

\begin{thm}[Gaidashev-Radu-Yampolsky]\label{continuous circle}
Let $\theta = \theta_* = (\sqrt{5}-1)/2$. Then there exists $\bar \epsilon >0$ such that for $b \in \bbD_{\bar\epsilon} \setminus \{0\}$, the boundary of the Siegel disk $\cD_b$ for $H_b$ is a homeomorphic image of the circle. In fact, the linearizing map
$$
\Psi_b : \bbD \times \{0\} \rightarrow \cD_b
$$
extends continuously and injectively (but not smoothly) to the boundary.
\end{thm}

Recall that $H_b$ is an automorphism of $\bbC^2$ with constant Jacobian $b \neq 0$. Hence, $H_b$ does not have any definite singularities that would obstruct the smoothness of $\partial \cD_b$ like in the one-dimensional case. Nonetheless, in the author's joint paper with Yampolsky (see \cite{YamY}), we proved that the Siegel boundary for a H\'enon map with golden-mean rotation number is not smooth:

\begin{thm}[Yampolsky-Y.]\label{nonsmooth circle}
Let $\theta = \theta_* = (\sqrt{5}-1)/2$. Then there exists $\bar \epsilon >0$ such that for $b \in \bbD_{\bar\epsilon} \setminus \{0\}$, the boundary of the Siegel disk $\cD_b$ for $H_b$ is not $C^1$-smooth.
\end{thm}

The properties of the Siegel boundary for H\'enon maps given in \thmref{continuous circle} and \thmref{nonsmooth circle} are also true for quadratic polynomials. Our main result states that the similarity between the one and two-dimensional case does not extend to quasisymmetry:

\begin{mainthm}[Non-Quasisymmetry]
Let $\theta = \theta_* = (\sqrt{5}-1)/2$. Then there exists $\bar\epsilon >0$ such that the set of parameter values $b$ for which the boundary of the Siegel disk $\cD_b$ for $H_b$ has unbounded geometry contains a dense $G_\delta$ subset in the disc $\mathbb{D}_{\bar\epsilon} \setminus \{0\}$.
\end{mainthm}

The proof of the Main Theorem follows the strategy used by de Carvalho, Lyubich and Martens in \cite{dCLMa} to obtain the analogous result for the limit Cantor sets of Feigenbaum H\'enon maps.


\section{Preliminaries}\label{preliminaries}

In this section, we provide a brief summary of the renormalization theory of semi-Siegel H\'enon maps. See \cite{Y} for complete details.

Let $b \in \bbD_{\bar \epsilon} \setminus \{0\}$ for some $\bar\epsilon>0$ sufficiently small. Consider the H\'enon map
$$
H_b(x,y)=\begin{bmatrix}
x^2+c_b - by\\
x
\end{bmatrix}
$$
that has a Siegel disc $\cD_b \subset \bbC^2$ with rotation number $\theta_* = (\sqrt{5}-1)/2$.

\subsection{Definition of Renormalization}

Let $\hat \Omega_0$ and $\hat \Gamma_0$ be suitably chosen topological bidisks in $\bbC^2$ such that $\hat \Omega_0 \cap \hat \Gamma_0 \ni (0,0)$ and $\hat \Omega_0 \cup \hat \Gamma_0 \supset \partial \cD_b$. The {\it pair representation} of $H_b$ is given by
$$
\hat \Sigma_0 = (\hat A_0, \hat B_0) := (H_b|_{\hat\Omega_0}, H_b|_{\hat\Gamma_0}).
$$

Let
\begin{equation}\label{eq:initial normalization}
\Phi_0(x,y) := (\lambda_0 x, \lambda_0 y),
\end{equation}
where $\lambda_0 := c_b$. Observe that
$$
\Phi_0^{-1} \circ H_b \circ \Phi_0(0,0) =(1, 0).
$$
The {\it normalized pair representation} of $H_b$ is defined as
$$
\Sigma_0 = (A_0, B_0) := (\Phi_0^{-1} \circ \hat A_0 \circ \Phi_0, \Phi_0^{-1} \circ \hat B_0 \circ \Phi_0).
$$
We may assume that for some topological discs $Z, W, V \subset \bbC$ containing $0$, the domains of $A_0$ and $B_0$ are given by
$$
\Omega = Z \times V := \Phi_0^{-1}(\hat \Omega_0)
\hspace{5 mm} \text{and} \hspace{5mm}
\Gamma = W \times V := \Phi_0^{-1}(\hat \Gamma_0)
\hspace{5 mm} \text{respectively.}
$$

The $n$th {\it renormalization} of $H_b$:
$$
\bfR^n(H_b) := \Sigma_n = (A_n, B_n),
$$
where
$$
A_n(x,y) = \begin{bmatrix}
a_n(x,y) \\
h_n(x,y)
\end{bmatrix}
\hspace{5mm} \text{and} \hspace{5mm}
B_n(x,y) = \begin{bmatrix}
b_n(x,y) \\
x
\end{bmatrix},
$$
is the pair of rescaled iterates of $\Sigma_0$ defined inductively as follows. Denote
$$
(a_n)_y(x) := a_n(x,y),
$$
and let
$$
H_{n+1}(x,y) :=
\begin{bmatrix}
(a_n)_y^{-1}(x) \\
y
\end{bmatrix}.
$$
Consider the non-linear change of coordinates defined as
$$
\Phi_{n+1} := H_{n+1} \circ \Lambda_{n+1},
$$
where
\begin{equation}\label{eq:affine rescaling}
\Lambda_{n+1}(x,y) = (\lambda_{n+1} x + c_{n+1}, \lambda_{n+1} y)
\end{equation}
is an affine rescaling map to be specified later. The pair $\Sigma_{n+1} = (A_{n+1}, B_{n+1})$ is defined as
$$
A_{n+1} = \Phi_{n+1}^{-1} \circ B_n \circ A_n^2 \circ \Phi_{n+1}
\hspace{5mm} \text{and} \hspace{5mm}
B_{n+1} = \Phi_{n+1}^{-1} \circ B_n \circ A_n \circ \Phi_{n+1}.
$$

\begin{lem}\label{degeneration}
For $n \geq 0$, let $\Sigma_n = (A_n, B_n)$ be the $n$th renormalization of $H_b$. Then $A_n$ and $B_n$ are bounded analytic maps that are well-defined on $\Omega = Z \times V$ and $\Gamma = W \times V$ respectively. Moreover, the dependence of $\Sigma_n$ on $y$ decays \emph{super-exponentially} fast. That is, we have
$$
\sup_{(x,y) \in \Omega}\|\partial_y A_n(x,y)\| < C\bar\epsilon^{2^n}
\hspace{5mm} \text{and} \hspace{5mm}
\sup_{(x,y) \in \Gamma}\|\partial_y B_n(x,y)\| < C\bar\epsilon^{2^n}.
$$
for some uniform constant $C >0$.
\end{lem}

The one-dimensional projections of $A_n$, $B_n$ and $\Sigma_n$ are given by
$$
\eta_n(x) := a_n(x, 0)
\hspace{5mm} , \hspace{5mm}
\xi_n(x) := b_n(x,0)
\hspace{5mm} \text{and} \hspace{5mm}
\zeta_n := (\eta_n, \xi_n)
\hspace{5mm} \text{respectively}.
$$
By \lemref{degeneration}, we see that $\eta_n$ and $\xi_n$ are bounded analytic functions defined on $Z$ and $W$ respectively. Moreover, the dynamics of $\Sigma_n$ degenerates to that of $\zeta_n$ super-exponentially fast. It is shown in \cite{Y} that $\eta_n$ and $\xi_n$ each have a unique simple critical point which are $C\bar\epsilon^{2^n}$-close to each other. We choose the normalizing constants $\lambda_n$ and $c_n$ in \eqref{eq:affine rescaling} so that
$$
\xi_n(0) =1
\hspace{5mm} \text{and} \hspace{5mm}
\xi_n'(0) = 0.
$$

\subsection{Renormalization Convergence}

Let $\zeta_* = (\eta_*, \xi_*)$ be the fixed point of the one-dimensional renormalization operator $\cR$ given in \cite{GaYam}. In particular, we have
\begin{equation}\label{eq:1d limit}
\lambda_*^{-1} \eta_* \circ \xi_* \circ \eta_* (\lambda_* x) = \eta_*(x)
\hspace{5mm} \text{and} \hspace{5mm}
\lambda_*^{-1} \eta_* \circ \xi_* (\lambda_* x)=\xi_*(x),
\end{equation}
where
$$
\lambda_* := \eta_* \circ \xi_* (0) \in \bbD
$$
is the universal scaling factor.

Convergence under renormalization for semi-Siegel H\'enon maps was first obtained in \cite{GaYam}. For the renormalization operator $\bfR$ defined above, the proof of convergence is given in \cite{Y}.

\begin{thm}\label{convergence}
As $n \to \infty$, we have the following convergences (each of which occurs at a geometric rate):
\begin{enumerate}[(i)]
\item $\zeta_n =(\eta_n, \xi_n) \to \zeta_* = (\eta_*, \xi_*)$;
\item $\lambda_n \to \lambda_*$; and
\item $\Phi_n \to \Phi_*$, where 
$$
\Phi_*(x,y) =
\begin{bmatrix}
\phi_*(x) \\
\lambda_* y
\end{bmatrix}
:=
\begin{bmatrix}
\eta_*^{-1}(\lambda_* x) \\
\lambda_* y
\end{bmatrix}.
$$
\end{enumerate}
\end{thm}

\subsection{Renormalization Limit Set}

Define the {\it $n$th microscope map of depth $k$} by
$$
\Phi_n^{n+k} := \Phi_{n+1} \circ \Phi_{n+2} \circ \ldots{} \circ \Phi_{n+k}.
$$

Let
$$
\Omega_n^{n+k} := \Phi_n^{n+k}(\Omega)
\hspace{5mm} \text{and} \hspace{5mm}
\Gamma_n^{n+k} := \Phi_n^{n+k}(\Gamma).
$$
Observe that $\{\Omega_n^{n+k} \cup \Gamma_n^{n+k}\}_{k=0}^\infty$ is a nested sequence of open sets. See \figref{fig:microscope}

\begin{prop}\label{2dcap}
Let $\lambda_* \in \bbD$ be the universal scaling factor. Then for all $0 \leq n, k$, we have
$$
\text{\emph{diam}}(\Omega_n^{n+k} \cup \Gamma_n^{n+k}) = O(|\lambda_*|^k).
$$
Consequently, there exists a point $(\kappa_n, 0) \in \Omega$, called the \emph{$n$th cap}, such that
$$
\bigcap_{k=0}^\infty \Omega_n^{n+k} \cup \Gamma_n^{n+k}= (\kappa_n,0).
$$
Moreover, $\kappa_n$ converges to $1$ geometrically fast as $n \to \infty$.
\end{prop}

\begin{remark}
The cap is a dynamically defined point with the same combinatorial address as the critical value $\xi_*(0) = 1$. In \cite{dCLMa}, the analog of the critical value is referred to as the {\it tip}.
\end{remark}

\begin{figure}[h]
\centering
\includegraphics[scale=0.25]{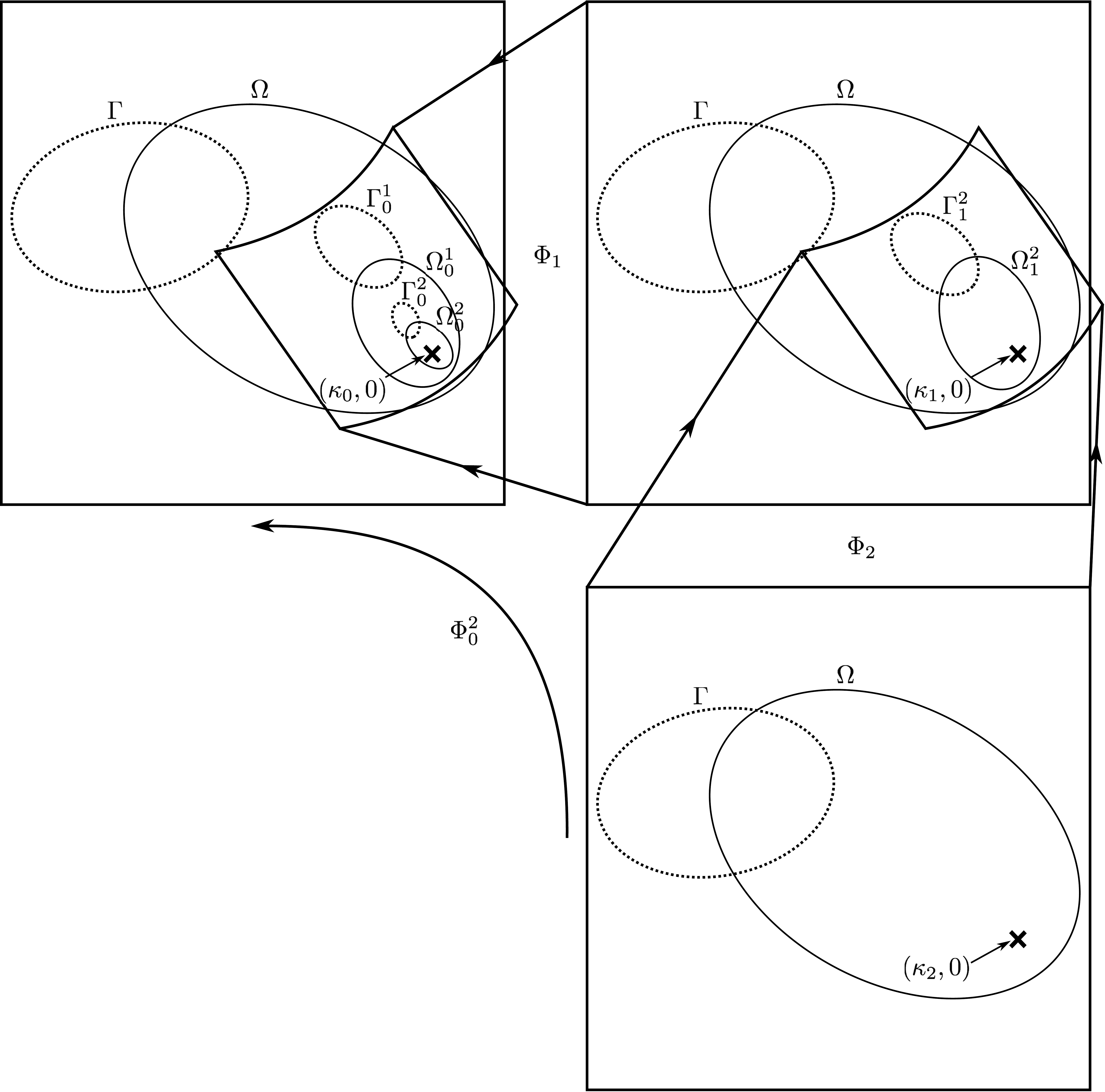}
\caption{The renormalization microscope map $\Phi_0^2$ obtained by composing the non-linear changes of coordinates $\Phi_1$ and $\Phi_2$. We have $\Omega_0^1 = \Phi_1(\Omega)$, $\Gamma_0^1 = \Phi_1(\Gamma)$, $\Omega_0^2 =\Phi_0^2(\Omega)$, $\Gamma_0^2 = \Phi_0^2(\Gamma)$, and $(\kappa_0,0)=\Phi_1((\kappa_1,0)) = \Phi_0^2((\kappa_2,0))$.}
\label{fig:microscope}
\end{figure}

Denote
$$
\hat\Omega_n := \Phi_0 \circ \Phi_0^n(\Omega)
\hspace{5mm} \text{and} \hspace{5mm}
\hat\Gamma_n := \Phi_0 \circ \Phi_0^n(\Gamma),
$$
where $\Phi_0$ is given in \eqref{eq:initial normalization}.

\begin{lem}
Let $\{q_n\}_{n=0}^\infty$ be the Fibonacci sequence. Define
$$
\hat A_n := H_b^{q_{2n+1}}|_{\hat \Omega_n}
\hspace{5mm} \text{and} \hspace{5mm}
\hat B_n := H_b^{q_{2n}}|_{\hat \Gamma_n}.
$$
Then the $n$th renormalization $\Sigma_n = (A_n, B_n)$ is given by
$$
A_n = (\Phi_0^n)^{-1} \circ \Phi_0^{-1} \circ \hat A_n \circ \Phi_0 \circ \Phi_0^n,
\hspace{5mm} \text{and} \hspace{5mm}
B_n = (\Phi_0^n)^{-1} \circ \Phi_0^{-1} \circ \hat B_n \circ \Phi_0 \circ \Phi_0^n.
$$
\end{lem}

In \cite{GaRaYam}, Gaidashev, Radu and Yampolsky showed that the renormalization limit set for a semi-Siegel H\'enon map coincides with its Siegel boundary:

\begin{thm}\label{limit curve}
For $n \in \mathbb{N}$, let
$$
X_n := \bigcup_{i = 0}^{q_{2n+1}-1} H_b^i(\hat\Omega_n)
\hspace{5mm} \text{and} \hspace{5mm}
Y_n := \bigcup_{i = 0}^{q_{2n}-1} H_b^i(\hat\Gamma_n).
$$
Then the Siegel boundary for $H_b$ is given by
$$
\partial \cD_b := \bigcap_{n=1}^\infty X_n \cup Y_n.
$$
\end{thm}


\section{Universality}

The proof of the main theorem involves giving precise estimates for various geometric quantities that arise when analyzing the dynamics of the semi-Siegel H\'enon map $H_b$ near its Siegel boundary $\partial \cD_b$. In order to do this, we first need an explicit description for the $n$th renormalization $\Sigma_n = (A_n, B_n)$ of $H_b$. In \cite{Y}, the author showed that $\Sigma_n$ has a {\it universal} first-order approximation in terms of its Jacobian. In this section, we strengthen this result to better suit our application.

As before, we write
$$
A_n(x,y) = \begin{bmatrix}
a_n(x,y)\\
h_n(x,y)
\end{bmatrix}
\hspace{5mm} \text{and} \hspace{5mm}
B_n(x,y) = \begin{bmatrix}
b_n(x,y)\\
x
\end{bmatrix}.
$$
Recall that $A_n$ and $B_n$ represent the $q_{2n+1}$th and $q_{2n}$th iterate of $H_b$ respectively. Accordingly, we expect the Jacobian of $A_n$ and $B_n$ to be on the order of
$$
 b^{q_{2n+1}} = \Jac H_b^{q_{2n+1}}
\hspace{5mm} \text{and} \hspace{5mm}
b^{q_{2n}} = \Jac H_b^{q_{2n}}
\hspace{5mm} \text{respectively}.
$$

The following theorem is a simplification of Theorem 7.3 in \cite{Y} in the case when the map being renormalized has constant Jacobian. In the general, non-constant Jacobian case, the dependence of $b_n$ on $y$ has a factor of $e^{r_n}$ for some uniformly bounded sequence $\{r_n\}_{n=0}^\infty \subset \bbC$. This is to account for small fluctuations caused by variation in the Jacobian along the renormalization limit set.

\begin{thm}\label{universalityb}
For $n \geq 0$, we have
$$
B_n(x,y)
=\begin{bmatrix}
\xi_n(x) - b^{q_{2n}} \, \beta(x) \, y \, (1+O(\rho^n))\\
x
\end{bmatrix},
$$
where $0 < \rho <1$ is a uniform constant; and $\beta(x)$ is a universal function that is uniformly bounded away from $0$ and $\infty$, and has a uniformly bounded derivative and distortion.
\end{thm}

Observe that
\begin{align}
\Jac B_n(x,y) &= \begin{vmatrix}
\xi_n'(x) - b^{q_{2n}} \, \beta'(x) \, y & -b^{q_{2n}} \, \beta(x)\\
1 & 0
\end{vmatrix}
+
\text{(higher order terms)} \nonumber \\
&= -b^{q_{2n}} \, \beta(x) + \text{(higher order terms)} \label{eq:jacb}
\end{align}
as we expected.

Next, we give a first-order approximation of $A_n$. In \cite{Y}, this is only done for the first coordinate $a_n$ (see Corollary 7.4).

\begin{thm}\label{universalitya}
For $n \geq 0$, we have
$$
A_n(x,y) =
\begin{bmatrix}
\eta_n(x) - b^{q_{2n}} \, \alpha(x) \, y \, (1+O(\rho^n))\\
\lambda_n^{-1} \check \eta_{n-1}(\lambda_n x) - b^{q_{2n}} \, \check{\alpha}(x) \, y \, (1+O(\rho^n))
\end{bmatrix},
$$
where $0 < \rho <1$ is a uniform constant; and
$$
\check \eta_k(x) = \eta_k(x) + O(|b|^{q_{2k}})
\hspace{3mm} , \hspace{3mm}
\alpha(x) := \frac{\eta_*'(x)}{\xi_*'(x)}\beta(x)
\hspace{3mm} \text{and} \hspace{3mm}
\check{\alpha}(x) := \frac{\eta_*'(\lambda_* x)}{\eta_*'(x)} \alpha(x).
$$
\end{thm}

\begin{proof}
Recall that
\begin{equation}\label{eq:adef}
A_{n+1} :=  \Lambda_{n+1}^{-1} \circ H_{n+1}^{-1} \circ B_n \circ A_n^2 \circ H_{n+1} \circ \Lambda_{n+1}.
\end{equation}
Denote
$$
(\tilde{x}, \tilde{y}) := \Lambda_{n+1}(x, y) = (\lambda_{n+1}x+c_{n+1}, \lambda_{n+1}y).
$$
It is not difficult to show, using \thmref{universalityb}, that $c_{n+1} = O(|b|^{q_{2n}})$. Moreover, we have
$$
A_n \circ H_{n+1}(x,y) = \begin{bmatrix}
x\\
\tilde{h}_n(x,y)
\end{bmatrix},
$$
where
$$
\tilde{h}_n(x,y) := h_n((a_n)_y^{-1}(x), y).
$$
Hence,
\begin{equation}\label{eq:xdependence}
h_{n+1}(x, y) = \lambda_{n+1}^{-1}a_n(\tilde{x}, \tilde{h}_n(\tilde{x}, \tilde{y})) = \lambda_{n+1}^{-1}\eta_n(\lambda_{n+1}x) + O(|b|^{q_{2n}}).
\end{equation}

Since $A_n$ and $B_n$ commute, we may rearrange the terms in \eqref{eq:adef} to obtain
$$
A_{n+1} =  \Lambda_{n+1}^{-1} \circ H_{n+1}^{-1} \circ A_n \circ H_{n+1} \circ \Lambda_{n+1} \circ B_{n+1}.
$$
Denote
$$
\tilde{\xi}_{n+1}(x):= \lambda_{n+1}\xi_{n+1}(x)+c_{n+1}
$$
and
\begin{align*}
\tilde{b}_{n+1}(x,y) &:= \lambda_{n+1} b_{n+1}(x,y) + c_{n+1}\\
&= \tilde{\xi}_{n+1}(x) - \lambda_{n+1}b^{q_{2(n+1)}}\beta(x)y(1+O(\rho^{n+1})).
\end{align*}
Then
$$
H_{n+1} \circ \Lambda_{n+1} \circ B_{n+1}(x,y) = \begin{bmatrix}
(a_n)_{\tilde x}^{-1}(\tilde b_{n+1}(x,y))\\
\tilde x
\end{bmatrix}.
$$
Neglecting higher order terms, $(a_n)_{\tilde x}^{-1}(\tilde b_{n+1}(x,y))$ has the same $y$-dependence as
$$
\eta_n^{-1}(\tilde b_{n+1}(x,y)) \approx \eta_n^{-1}(\tilde \xi_{n+1}(x)) - (\eta_n^{-1})'(\tilde \xi_{n+1}(x)) \lambda_{n+1} b^{q_{2(n+1)}} \beta(x)y
$$
Let
$$
\check{x} := \lambda_n\eta_n^{-1}(\tilde \xi_{n+1}(x)).
$$
From \eqref{eq:xdependence}, and again neglecting higher order terms, we see that the second coordinate $h_n\big((a_n)_{\tilde x}^{-1}(\tilde b_{n+1}(x,y)), \tilde x\big)$ has the same $y$-dependence as
$$
\lambda_n^{-1}\eta_{n-1}\big(\lambda_n\eta_n^{-1}(\tilde b_{n+1}(x,y))\big)
\approx
\lambda_n^{-1}\eta_{n-1}(\check{x}) - \eta_{n-1}'(\check{x})(\eta_n^{-1})'(\tilde \xi_{n+1}(x)) \lambda_{n+1} b^{q_{2(n+1)}} \beta(x)y.
$$
Lastly, we compute
\begin{align*}
\lim_{n\to \infty}\eta_{n-1}'(\check{x})(\eta_n^{-1})'(\tilde \xi_{n+1}(x))
&=
\eta_*' \big( \lambda_* \eta_*^{-1}(\lambda_*\xi_*(x))\big)(\eta_*^{-1})'(\lambda_* \xi_*(x))\\
&=
\eta_*' \big(\eta_* \circ \xi_*(\lambda_*^2x))\big)(\eta_*^{-1})'(\eta_* \circ \xi_*(\lambda_*x))\\
&=
\frac{\eta_*' \big(\eta_* \circ \xi_* (\lambda_*^2 x) \big)}{\eta_*'(\xi_*(\lambda_* x))}\\
&=
\frac{\eta_*'( \lambda_* x)}{\xi_*'(\lambda_* x)} \cdot \frac{\xi_*'(\lambda_* x)}{\xi_*'(x)}\\
&= \frac{\eta_*'( \lambda_* x)}{\xi_*'(x)},
\end{align*}
where in the second equality, we used \eqref{eq:1d limit} and in the third equality, we used the derivative of \eqref{eq:1d limit}. The result follows.
\end{proof}

Observe:
\begin{align*}
\Jac A_n(x,y) &= \begin{vmatrix}
\eta_n'(x) - b^{q_{2n}} \, \alpha'(x) \, y & -b^{q_{2n}} \, \alpha(x)\\
\check{\eta}_{n-1}'(\lambda_n x) - b^{q_{2n}} \, \check{\alpha}'(x) \, y & -b^{q_{2n}} \, \check{\alpha}(x)
\end{vmatrix}
+
\text{(higher order terms)}\\
&= \big(-\eta_n'(x) \, \check{\alpha}(x)+ \check{\eta}_{n-1}'(\lambda_n x) \, \alpha(x)\big)b^{q_{2n}} + \text{(higher order terms)}\\
&= \big(-\eta_*'(x) \, \check{\alpha}(x)+ \eta_*'(\lambda_* x) \, \alpha(x)\big)b^{q_{2n}} + \text{(higher order terms)}\\
& = \text{(higher order terms)}.
\end{align*}
Hence, we see that the first-order approximation of $A_n$ is not precise enough to ``see'' the Jacobian of $A_n$, which we expect to be of higher order ($b^{q_{2n+1}}$ rather than $b^{q_{2n}}$). Fortunately, the approximation can be made much more precise by considering $A_n \circ \Phi_{n+1}$ instead of $A_n$.

\begin{thm}\label{universalityaphi}
For $n \geq 0$, we have
$$
A_n \circ \Phi_{n+1}(x,y) =
\begin{bmatrix}
\lambda_{n+1}x+c_{n+1}\\
\check{\xi}_n^{-1}(\lambda_{n+1}x) - b^{q_{2n+1}} \, \chi(x)\, \lambda_{n+1}\, y(1+O(\rho^n))
\end{bmatrix},
$$
where $0 < \rho <1$ is a uniform constant; and
$$
\check \xi_n(x) = \xi_n(x) + O(|b|^{q_{2(n-1)}})
\hspace{3mm} \text{and} \hspace{3mm}
\chi(x) := \frac{\xi_*'(\lambda_* x)}{\xi_*'(x)}\frac{\beta(x)}{\beta(\lambda_* x)} = \frac{\eta_*'(\lambda_* x)}{\eta_*'(x)}\frac{\alpha(x)}{\alpha(\lambda_* x)}.
$$
\end{thm}

\begin{proof}
Write
$$
A_n \circ H_{n+1}(x,y) = \begin{bmatrix}
x\\
\tilde{h}_n(x,y)
\end{bmatrix},
$$
where
$$
\tilde{h}_n(x,y) := h_n((a_n)_y^{-1}(x), y).
$$
By \thmref{universalitya}, we have
\begin{equation}\label{eq:tildeh}
\tilde{h}_n(x,y) = \lambda_n^{-1} \eta_{n-1}(\lambda_n\eta_n^{-1}(x)) + O(|b|^{q_{2(n-1)}}).
\end{equation}

By definition, $\eta_n(x) = a_n(x,0)$, where $a_n$ is the first coordinate of $A_n$. Since $A_{n-1}$ and $B_{n-1}$ commute, we have
$$
A_n := \Phi_n^{-1} \circ B_{n-1} \circ A_{n-1}^2 \circ \Phi_n= B_n \circ \Phi_n^{-1} \circ A_{n-1} \circ \Phi_n.
$$
Thus,
\begin{equation}\label{eq:etan}
\eta_n = \xi_n(\lambda_n^{-1}\eta_{n-1}(\lambda_n x)) + O(|b|^{q_{2(n-1)}}).
\end{equation}
Taking the inverse of \eqref{eq:etan} and plugging into \eqref{eq:tildeh}, we obtain
$$
\lambda_n^{-1} \eta_{n-1}(\lambda_n\eta_n^{-1}(x)) = \xi_n^{-1}(x) + O(|b|^{q_{2(n-1)}}).
$$

To derive the first-order approximation of the $y$-dependence of $A_n \circ \Phi_{n+1}$, consider
$$
B_{n+1} = \Phi_{n+1}^{-1} \circ B_n \circ A_n \circ \Phi_{n+1}.
$$
Taking the Jacobian, we obtain
\begin{equation}\label{eq:jaca0}
\Jac_{\bfx_0}(B_{n+1}) = \Jac_{\bfx_2}(\Phi_{n+1}^{-1}) \cdot \Jac_{\bfx_1}(B_n) \cdot \Jac_{\bfx_0} (A_n \circ \Phi_{n+1}),
\end{equation}
where
$$
\bfx_0:= (x,y)
\hspace{5mm} , \hspace{5mm}
\bfx_1 =(x_1, y_1) := A_n \circ \Phi_{n+1}(\bfx_0)
\hspace{5mm} \text{and} \hspace{5mm}
\bfx_2 =(x_2, y_2) := B_n(\bfx_1).
$$
Neglecting higher order terms, we have
$$
x_1 \approx \lambda_{n+1} x
\hspace{5mm} \text{and}\hspace{5mm}
x_2 \approx \xi_n(\lambda_{n+1}x).
$$
By \eqref{eq:jacb}, it follows that
\begin{equation}\label{eq:jaca1}
\Jac_{\bfx_0}(B_{n+1}) \approx b^{q_{2n+2}} \beta(x)
\hspace{5mm} \text{and} \hspace{5mm}
\Jac_{\bfx_1}(B_n) \approx b^{q_{2n}}\beta(\lambda_{n+1} x).
\end{equation}
Observe:
\begin{align}
\Jac_{\bfx_2} \Phi_{n+1}^{-1} &= \lambda_{n+1}^{-2}\begin{vmatrix}
\partial_x a_n(\bfx_2) & \partial_y a_n(\bfx_2)\\
0 & 1
\end{vmatrix}\nonumber\\
&\approx \lambda_{n+1}^{-2}\eta_n'(\xi_n(\lambda_{n+1} x))\label{eq:jaca2}.
\end{align}
Plugging in \eqref{eq:jaca1} and \eqref{eq:jaca2} into \eqref{eq:jaca0}, we obtain
\begin{equation}\label{eq:jaca3}
\Jac_{\bfx_0} (A_n \circ \Phi_{n+1}) \approx b^{q_{2n+1}}\lambda_{n+1}^2\frac{\beta(x)}{\beta(\lambda_{n+1} x)}\eta_n'(\xi_n(\lambda_{n+1} x)).
\end{equation}

Lastly, write:
$$
A_n \circ \Phi_{n+1}(x,y) = \begin{bmatrix}
\lambda_{n+1}x + c_{n+1}\\
\tilde{h}_n(x,0) + E_n(x,y)
\end{bmatrix},
$$
where $E_n(x,y)$ is undetermined. Since
$$
\Jac_{(x,y)} (A_n \circ \Phi_{n+1}) = \lambda_{n+1}\partial_y E_n(x,y),
$$
plugging in \eqref{eq:jaca3} and integrating both sides, we obtain the desired formula.
\end{proof}

Recall that the $n$th cap $(\kappa_n, 0) \in \Omega$ is a dynamically defined point in the renormalization limit set for $\Sigma_n$. It is not difficult to see that we have
$$
\Phi_n((\kappa_n, 0)) = (\kappa_{n-1}, 0)
\hspace{5mm} \text{and} \hspace{5mm}
\Phi_n^{n+k}((\kappa_{n+k}, 0)) = (\kappa_n, 0).
$$
Denote
$$
D_n:= D_{(\kappa_n,0)}\Phi_n
\hspace{5mm} \text{and} \hspace{5mm}
D_n^{n+k} := D_{(\kappa_{n+k},0)}\Phi_n^{n+k}.
$$
Observe
$$
D_n^{n+k} = D_{n+1} \cdot D_{n+2} \cdot \ldots{} \cdot D_{n+k}.
$$

The following estimates on the derivative the microscope maps at the cap is a corollary of \thmref{universalitya}. The statement is more precise than the one given in \cite{Y}, but it follows from the same proof.

\begin{thm}\label{cap derivative}
Write
$$
D_n = \begin{bmatrix}
1 & t_nb^{q_{2(n-1)}} \\
0 & 1
\end{bmatrix}
\begin{bmatrix}
u_n & 0 \\
0 & \lambda_n
\end{bmatrix}.
$$
Then there exist a uniform positive constant $\rho <1$ such that the following estimates hold for all $n \geq 1$:
\begin{enumerate}[(i)]
\item $u_n = \lambda_*^2(1+O(\rho^n))$,
\item $\lambda_n = \lambda_*(1+O(\rho^n))$, and
\item $t_n = \lambda_*\alpha(1)(1+O(\rho^n))$.
\end{enumerate}
Consequently, for $0 \leq n, k$, we have
$$
D_n^{n+k} =  \begin{bmatrix}
1 & t_n^{n+k}b^{q_{2n}}\\
0 & 1
\end{bmatrix}
\begin{bmatrix}
u_n^{n+k} & 0 \\
0 & \lambda_n^{n+k}
\end{bmatrix},
$$
where
\begin{enumerate}[(i)]
\item $u_n^{n+k} := u_{n+1} \cdot u_{n+1} \cdot \ldots \cdot u_{n+k} = \lambda_*^{2k}(1+O(\rho^n))$,
\item $\lambda_n^{n+k} := \lambda_{n+1} \cdot \lambda_{n+2} \cdot \ldots \cdot \lambda_{n+k} = \lambda_*^k(1+O(\rho^n))$, and
\item $t_n^{n+k} = t_{n+1}(1 +O(|b|^{q_{2n+1}}))$.
\end{enumerate}
\end{thm}


\section{Proof of Non-Quasisymmetry}

\begin{defn}
Let $\gamma$ be a continuous arc or a simple closed curve. For $x, y \in \gamma$, let $[x, y]_\gamma$ denote the smallest subarc of $\gamma$ with endpoints at $x$ and $y$. We say that $\gamma$ is $K$-{\it quasisymmetric} for some $K>0$ if
$$
\diam([x,y]_\gamma) < K \dist(x,y)
\hspace{5mm} \text{for all} \hspace{5mm}
x,y \in \gamma.
$$
If $\gamma$ is not $K$-quasisymmetric for all $K > 0$, then we say that $\gamma$ has {\it unbounded geometry}.
\end{defn}

\begin{proof}[Proof of the Main Theorem]
Let  $n, k \in \bbN$ be sufficiently large. To prove the desired result, we analyze the geometry of the Siegel boundary $\partial \cD_b$ in three different scales: that of $\Sigma_{n+k}$, $\Sigma_n$ and $\Sigma_0$.

We start in the scale of $\Sigma_{n+k}$. Consider the dynamically defined points
$$
\bfq_{n+k} := (\kappa_{n+k}, 0)
\hspace{5mm} \text{and} \hspace{5mm}
\bfp_{n+k} := A_{n+k}(\bfq_{n+k}).
$$
Next, in the scale of $\Sigma_n = (A_n, B_n)$, consider
$$
\bfq_n^{n+k} := A_n\circ \Phi_n^{n+k}(\bfq_{n+k}) = A_n((\kappa_n, 0))
\hspace{5mm} \text{and} \hspace{5mm}
\bfp_n^{n+k} := A_n \circ \Phi_n^{n+k}(\bfp_{n+k}).
$$
Lastly, in the scale of $\Sigma_0$, consider
$$
\bfq_0^{n+k} := \Phi_0^n(\bfq_n^{n+k})
\hspace{5mm} \text{and} \hspace{5mm}
\bfp_0^{n+k} := \Phi_0^n(\bfp_n^{n+k}).
$$
See \figref{fig:cusp}.

It is easy to see that
$$
\bfq_0^{n+k} = H_b^i((\kappa_0, 0)) \in \partial \cD_b
\hspace{5mm} \text{and} \hspace{5mm}
\bfp_0^{n+k} = H_b^j((\kappa_0, 0)) \in \partial \cD_b.
$$
for some $i, j \in \bbN$. Denote the $x$- and $y$-coordinate of $\bfq_{n+k} - \bfp_{n+k}$ by $\Delta_{n+k} x$ and $\Delta_{n+k} y$ respectively. The $x$- and $y$-coordinates of $\bfq_n^{n+k}-\bfq_n^{n+k}$ and $\bfq_0^{n+k} - \bfp_0^{n+k}$ are denoted similarly, but with $\Delta_n^{n+k}$ and $\Delta_0^{n+k}$ respectively instead of $\Delta_{n+k}$.

\begin{figure}[h]
\centering
\includegraphics[scale=0.25]{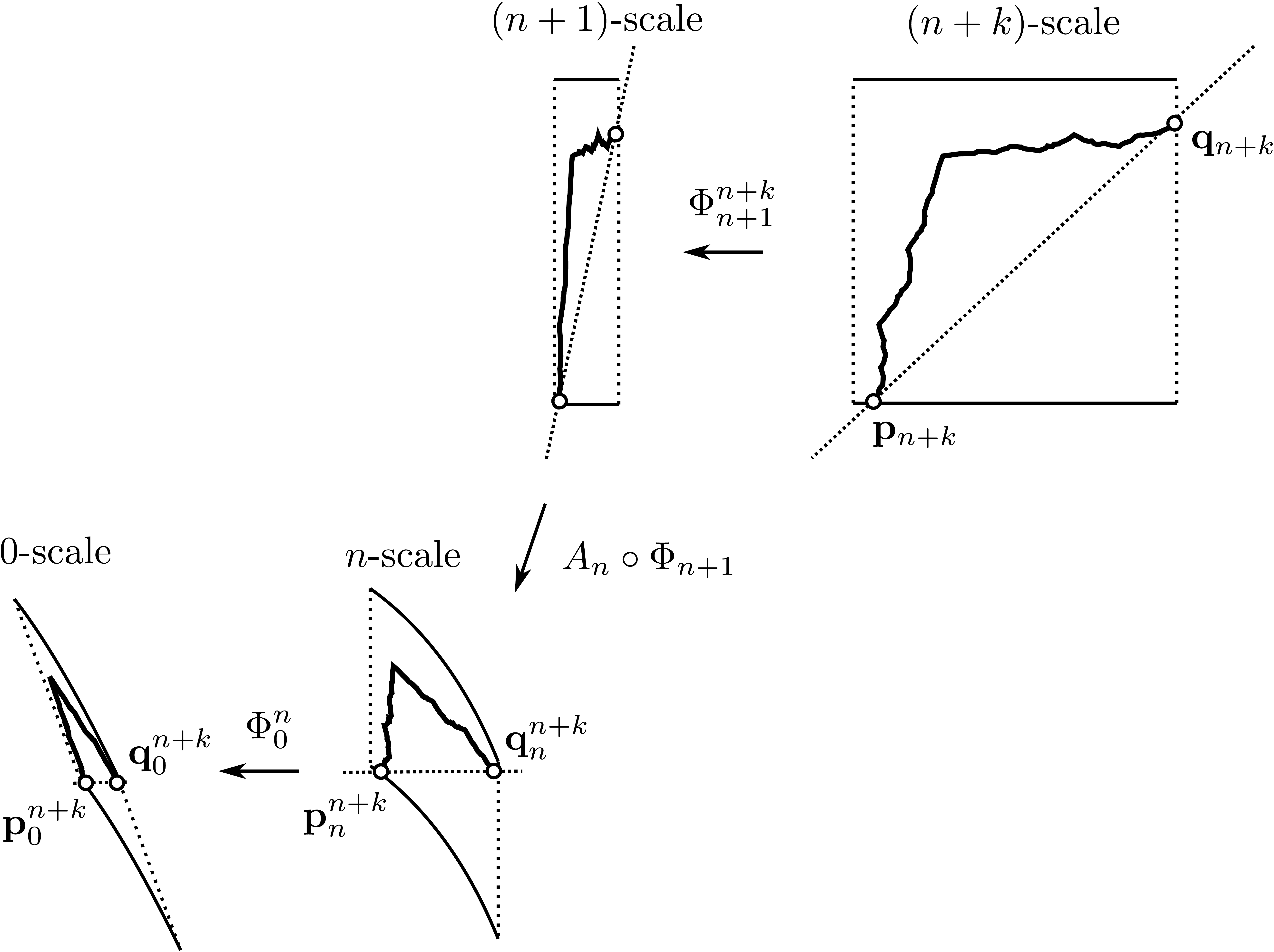}
\caption{Exploiting universality to create a cusp in the Siegel boundary $\partial \cD_b$.}
\label{fig:cusp}
\end{figure}

First, note that $\Delta_{n+k}\, x$ and $\Delta_{n+k}\, y$ converge to some uniform limits $\Delta_*\, x$ and $\Delta_*\, y$ respectively as $k \to \infty$. By \thmref{universalityaphi} and \ref{cap derivative}, we have:
\begin{equation}\label{eq:deltax}
\Delta_n^{n+k}\, x = \lambda_*^{2k-1}\Delta_*\, x(1+O(\rho^n))
\end{equation}
and
\begin{equation}\label{eq:deltay}
\Delta_n^{n+k}\, y =\lambda_*^{k-1}(\lambda_*^kC_1 - b^{q_{2n+1}}C_2)(1+O(\rho^n)),
\end{equation}
where
$$
C_1 := (\xi_*^{-1})'(\lambda_*)\Delta_*\, x
\hspace{5mm} \text{and} \hspace{5mm}
C_2 := \lambda_*\chi(1)\Delta_*\, y
$$
are uniform constants.

The values of $b$ that solve
$$
b^{q_{2n+1}} = \frac{C_1}{C_2}\lambda_*^k
$$
as $k$ and $n$ run through $\mathbb{N}$ is dense in $\mathbb{D}_{\bar\epsilon}$. Let $\mathbb{B}_m \subset \bbD_{\bar\epsilon}$ be the set of parameter values such that for some $n \geq m$ and $k \in \mathbb{N}$, we have
$$
|b|^{q_{2n+1}} \asymp |\lambda_*|^k
$$
and
\begin{equation}\label{eq:vertical0}
\Delta_n^{n+k}\, y = O(|\lambda_*|^{n+2k}).
\end{equation}
Then $\mathbb{B}_m$ is a dense open set of $\mathbb{D}_{\bar\epsilon}$. Hence, the intersection
$$
\mathbb{B}_\infty := \bigcap_{m = N}^\infty \mathbb{B}_m
$$
is a $G_\delta$-subset of $\mathbb{D}_{\bar\epsilon}$, consisting of parameters $b$ for which \eqref{eq:vertical0} holds for infinitely many $n \in \bbN$.

Assume that $b \in \mathbb{B}_\infty$ and that \eqref{eq:vertical0} holds for $n$ and $k$. By \thmref{cap derivative}, we have:
$$
\Delta_0^{n+k}\, x \asymp \lambda_*^{2n}\Delta_n^{n+k}\, x \asymp \lambda_*^{2n+2k-1}
$$
and
$$
\Delta_0^{n+k}\, y \asymp \lambda_*^n\Delta_n^{n+k}\, y = O(|\lambda_*|^{2n+2k})
$$
Hence:
\begin{equation}\label{eq:total dist}
\dist(\bfq_0^{n+k}, \bfp_0^{n+k}) = O(|\lambda_*|^{2n+2k}).
\end{equation}

We now estimate the diameter of the subarc $[\bfq_0^{n+k}, \bfp_0^{n+k}]_{\partial \cD_b}$. Consider the points
$$
\check{\bfq}_{n+k} := \Phi_{n+k+1}\circ A_{n+k+1}((\kappa_{n+k+1}, 0))
\hspace{5mm} , \hspace{5mm}
\check{\bfq}_n^{n+k} := A_n \circ \Phi_n^{n+k}(\check{\bfq}_{n+k})
$$
$$
\hspace{5mm} \text{and} \hspace{5mm}
\check{\bfq}_0^{n+k} := \Phi_0^n(\check{\bfq}_n^{n+k}).
$$
It is easy to see that $\check{\bfq}_0^{n+k} \in [\bfq_0^{n+k}, \bfp_0^{n+k}]_{\partial \cD_b}$. Denote the $x$- and $y$-coordinate of $\check{\bfq}_{n+k} - \bfq_{n+k}$ by $\check{\Delta}_{n+k}\, x$ and $\check{\Delta}_{n+k}\, y$ respectively. The $x$- and $y$-coordinates of $\check{\bfq}_n^{n+k}-\bfq_n^{n+k}$ and $\check{\bfq}_0^{n+k} - \bfq_0^{n+k}$ are denoted similarly, but with $\check{\Delta}_n^{n+k}$ and $\check{\Delta}_0^{n+k}$ respectively instead of $\Delta_{n+k}$.

Note that $\check{\Delta}_{n+k}\, x$ and $\check{\Delta}_{n+k}\, y$ converge to $\lambda_*^2 \Delta_*\, x$ and $\lambda_* \Delta_*\, y$ respectively as $k \to \infty$. By similar considerations as for \eqref{eq:deltax} and \eqref{eq:deltay}, we obtain
$$
\check{\Delta}_n^{n+k}\, x = \lambda_*^{2k+1}\Delta_*\, x(1+O(\rho^n))
$$
and
\begin{align*}
\check{\Delta}_n^{n+k}\, y &=\lambda_*^k(\lambda_*^{k+1}C_1 - b^{q_{2n+1}}C_2)(1+O(\rho^n))\\
&= \lambda_*^{2k}C_1\left(\lambda_* - \frac{b^{q_{2n+1}}C_2}{\lambda_*^kC_1}\right)(1+O(\rho^n)).
\end{align*}
Since $\lambda_* <1$, it follows from \eqref{eq:vertical0} that $\check{\Delta}_n^{n+k}\, y \asymp \lambda_*^{2k}$. Hence,
$$
\check{\Delta}_0^{n+k}\, y \asymp \lambda_*^n \check{\Delta}_0^{n+k} \asymp \lambda_*^{n+2k}.
$$
Therefore,
$$
\frac{\diam([\bfq_0^{n+k}, \bfp_0^{n+k}]_{\partial \cD_b})}{\dist(\bfq_0^{n+k}, \bfp_0^{n+k})} \geq \frac{\dist(\check{\bfq}_0^{n+k}, \bfq_0^{n+k})}{\dist(\bfq_0^{n+k}, \bfp_0^{n+k})} \asymp \lambda_*^{-n} \to \infty
\hspace{5mm} \text{as} \hspace{5mm}
n \to \infty.
$$
The result follows.
\end{proof}

\end{document}